\documentclass[a4paper,10pt]{amsart}
\usepackage[T1]{fontenc}
\usepackage[utf8]{inputenc}
\usepackage{lmodern}
\usepackage[english]{babel}

\usepackage{amsmath}
\usepackage{amssymb}
\usepackage{amsbsy}
\usepackage{mathtools}
\usepackage{mathrsfs}
\usepackage{bookmark}
\usepackage{hyperref}
\hypersetup{
	colorlinks=true,
	linkcolor=blue,
	filecolor=magenta,      
	urlcolor=cyan,
}

\usepackage[alphabetic]{amsrefs}

\usepackage{enumitem}
\usepackage{graphicx}
\usepackage{tikz}
\usepackage{tikz-cd}
\usetikzlibrary{matrix,arrows,decorations.pathmorphing}
\usepackage[all,cmtip]{xy}
\usepackage{comment}
\newcommand{\mypar}{~\\\noindent\refstepcounter{subsubsection}\textbf{(\thesubsubsection)}\space}
\newcommand\arr{\ifinner\to\else\longrightarrow\fi}
\newcommand{\thmref}[1]{\hyperref[#1]{Theorem \ref{#1}}}
\newcommand{\propref}[1]{\hyperref[#1]{Proposition \ref{#1}}}
\newcommand{\lmref}[1]{\hyperref[#1]{Lemma \ref{#1}}}
\newcommand{\corref}[1]{\hyperref[#1]{Corollary \ref{#1}}}
\newcommand{\rmkref}[1]{\hyperref[#1]{Remark \ref{#1}}}
\newcommand{\dfref}[1]{\hyperref[#1]{Definition \ref{#1}}}
\newcommand{\parref}[1]{\hyperref[#1]{(\ref{#1})}}

\newenvironment{thm}{\vspace{8.0pt plus 2.0pt minus 4.0pt}\noindent\refstepcounter{subsubsection}\textbf{(\thesubsubsection) Theorem.}\itshape}{\vspace{8.0pt plus 2.0pt minus 4.0pt}}
\newenvironment{prop}{\vspace{8.0pt plus 2.0pt minus 4.0pt}\noindent\refstepcounter{subsubsection}\textbf{(\thesubsubsection) Proposition.}\itshape}{\vspace{8.0pt plus 2.0pt minus 4.0pt}}
\newenvironment{lm}{\vspace{8.0pt plus 2.0pt minus 4.0pt}\noindent\refstepcounter{subsubsection}\textbf{(\thesubsubsection) Lemma.}\itshape}{\vspace{8.0pt plus 2.0pt minus 4.0pt}}
\newenvironment{cor}{\vspace{8.0pt plus 2.0pt minus 4.0pt}\noindent\refstepcounter{subsubsection}\textbf{(\thesubsubsection) Corollary.}\itshape}{\vspace{8.0pt plus 2.0pt minus 4.0pt}}

\newtheorem*{thm-no-num}{Theorem}
\newtheorem*{cor-no-num}{Corollary}

\newcommand{\Dcal}{\mathcal D}

\newcommand{\Fcal}{\mathcal F}

\newcommand{\Ical}{\mathcal I}
\newcommand{\Jcal}{\mathcal J}

\newcommand{\Lcal}{\mathcal L}
\newcommand{\Mcal}{\mathcal M}

\newcommand{\Ocal}{\mathcal O}

\newcommand{\Qcal}{\mathcal Q}

\newcommand{\Scal}{\mathcal S}
\newcommand{\Tcal}{\mathcal T}
\newcommand{\Ucal}{\mathcal U}

\newcommand{\Pic}{\text{\rm Pic}}
\newcommand{\GL}{\text{GL}}
\newcommand{\PGL}{\text{PGL}}
\newcommand{\Sym}{\text{Sym}}
\newcommand{\Spec}{\text{Spec}}

\newcommand{\pr}{\text{\rm pr}}

\newcommand{\im}{\text{Im}}

\newcommand\PP{\mathbb{P}}

\newcommand\ZZ{\mathbb{Z}}

\newcommand{\FF}{\boldsymbol{F}}
\newcommand{\GG}{\boldsymbol{G}}
\newcommand{\Gras}{\operatorname{Gr}_m(W_{d})}

\newcommand{\rat}{_{rat}}
\newcommand{\sm}{_{sm}}
\newcommand{\Gr}{\operatorname{Gr}}

\begin{document}
\title[Integral Picard group of some stacks of polarized K3 surfaces]{Integral Picard group of some stacks of polarized K3 surfaces of low degree}
\author[A. Di Lorenzo]{Andrea Di Lorenzo}
	\address[A. Di Lorenzo]{Humboldt Universit\"{a}t zu Berlin, Germany}
	\email{andrea.dilorenzo@hu-berlin.de}
\date{\today}
\begin{abstract}
    We compute the integral Picard group of the stack $\Mcal_{2l}$ of polarized K3 surfaces with at most rational double points of degree $2l=4,6,8$. We show that in this range the integral Picard group is torsion-free and that a basis is given by certain elliptic Noether-Lefschetz divisors together with the Hodge line bundle. 
    
    To achieve this result, we investigate certain stacks of complete intersections and their Picard groups by means of equivariant geometry.
    
    In the end we compute an expression of the class of some Noether-Lefschetz divisors, restricted to an open substack of $\Mcal_{2l}$, in terms of the basis mentioned above.
\end{abstract}
\maketitle

\section*{Introduction}
Picard groups of moduli problems are the subject of an extensive mathematical literature. Since the landmark paper \cite{Mum63}, where Mumford computed the Picard group of the stack $\Mcal_{1,1}$ of elliptic curves, many explicit presentations of Picard groups have been obtained: among the most striking results, we mention here the marvellous papers \cite{Har} and \cite{AC}, respectively by Harer and Arbarello, Cornalba, where the authors determined  the Picard group of the stack $\Mcal_{g,n}$ of smooth $n$-pointed curves of genus $g$ for $g\geq 3$, as well as the Picard group of its compactification $\overline{\Mcal}_{g,n}$ by means of stable pointed curves.

Another moduli space whose Picard group had been thoroughly investigated in the last years is the moduli space $M_{2l}$ of primitively polarized K3 surfaces with at most rational double points of degree $2l$, using approaches that involved intersection theory, hodge-theoretic methods and modular forms.

For instance, in \cite{OG} it has been proved that the rank of $\Pic(M_{2l})\otimes \mathbb{Q}$ can be arbitrarily large.

In \cite{Bru} the author computed the rank of an interesting sub-vector space of $\Pic(M_{2l})\otimes\mathbb{Q}$, namely the linear subspace of the so called Noether-Lefschetz divisors.

Furthermore, in \cite{MP} the authors conjectured that the linear subspace of Noether-Lefschetz divisors is actually equal to the whole rational Picard group of $M_{2l}$. This conjecture is now a theorem (see \cite{LT}, \cite{GLT}, \cite{BLMM}).

In this work we investigate the \emph{integral} Picard group of the \emph{stack} $\Mcal_{2l}$ of primitively polarized K3 surfaces with at most rational double points of degree $2l$ for $l=2,3,4$. Our main result is the following:
\begin{thm-no-num}
    Let $\Dcal_{d,h}$ indicate the Noether-Lefschetz divisors on the stack of primitively polarized K3 surfaces with at most rational double points (see \parref{par:NL}) and let $\lambda_1$ be the class of the Hodge line bundle (see \parref{par:hodge line bundle}). Then:
    \begin{itemize}
        \item $\Pic(\Mcal_4)\simeq \ZZ\cdot \lambda_1 \oplus \ZZ\cdot [\Dcal_{1,1}] \oplus \ZZ\cdot [\Dcal_{2,1}]$.
        \item $\Pic(\Mcal_6)\simeq \ZZ\cdot \lambda_1 \oplus \ZZ\cdot [\Dcal_{1,1}] \oplus \ZZ\cdot [\Dcal_{2,1}] \oplus \ZZ\cdot [\Dcal_{3,1}]$.
        \item $\Pic(\Mcal_8)\simeq \ZZ\cdot \lambda_1 \oplus \ZZ\cdot [\Dcal_{1,1}] \oplus \ZZ\cdot [\Dcal_{2,1}]\oplus \ZZ\cdot [\Dcal_{3,1}]$.
    \end{itemize}
\end{thm-no-num}
In particular, all these groups turn out to be torsion-free.
We also compute explicitly some classes of Noether-Lefschetz divisors in terms of the generators above.
\begin{thm-no-num}
    Let $\Ucal_{2l}$ be the open substack of primitively polarized K3 surfaces with at most rational double points of degree $2l$ whose polarization induces a birational morphism. Then:
    \begin{itemize}
        \item In $\Pic(\Ucal_4)\simeq\ZZ\cdot\lambda_1$ we have:
        \[ [\Dcal_{0,0}]=108\cdot\lambda_1,\quad [\Dcal_{3,1}]=320\cdot\lambda_1 \]
        \item In $\Pic(\Ucal_6)\simeq\ZZ\cdot \lambda_1\oplus\ZZ\cdot[\Dcal_{3,1}]$ we have:
        \[ [\Dcal_{0,0}]=78\cdot [\Dcal_{3,1}]+98\cdot\lambda_1 \]
        \item In $\Pic(\Ucal_8\setminus\Dcal_{3,1})\simeq\ZZ\cdot\lambda_1$ we have:
        \[ [\Dcal_{0,0}]=80\cdot\lambda_1 \]
    \end{itemize}
\end{thm-no-num}
These results are achieved by giving a presentation as quotient stacks for some substacks of $\Mcal_{2l}$ (where  $l=2,3,4$).

In particular, we show that these substacks are isomorphic to certain stacks of complete intersections, of which we are able to compute the Picard group using equivariant techniques.

\subsection*{Structure of the paper} 
Here we briefly present the content of each section of the paper. A more detailed account can be found at the beginning of every section.

    In Section \ref{sec:complete intersections} we define two types of stack of complete intersections, namely $\GG(d,m,n)$ and $\FF(a,b,n)$. Of both these stacks we give a presentation as quotient stacks and we compute their integral Picard groups.
    
    In Section \ref{sec:recap} we recall some notions from the theory of polarized K3 surfaces with at most rational double points and we prove some results which will be relevant in the remainder of the paper.
    
    In Section \ref{sec:Pic} we compute the Picard groups of $\Mcal_{2l}$, the stack of polarized K3 surfaces with at most rational double points of degree $2l$, for $l=2,3,4$ (see \thmref{thm:Pic K4}, \thmref{thm:Pic(K6)} and \thmref{thm:Pic K8}).
    
    In Section \ref{sec:computations} we compute the classes of some Noether-Lefschetz divisors in terms of the generators of the Picard groups.

\section{Stacks of complete intersections}\label{sec:complete intersections}
In this section we introduce two stacks: the first one, denoted $\GG(d,m,n)$ (see \parref{par:GG}), is the stack whose objects are flat families of complete intersections of $m$ hypersurfaces of degree $d$ inside a Brauer--Severi variety having $n$-dimensional fibres.

The second one, denoted $\FF(a,b,n)$ (see \parref{par:FF}), is the stack whose objects are families of complete intersections of two hypersurfaces of degree $a$ and $b$ inside a Brauer-Severi variety having $n$-dimensional fibres.

We give a presentation of these stacks as quotients (\propref{prop:GG is a quotient} and \propref{prop:FF is a quotient}) and we compute their Picard groups (\propref{prop:Pic GG} together with \propref{prop:Pic FF}).
\subsection{Equidegree complete intersections}
\mypar \label{par:GG} Let $n$, $d$, $m$ be three strictly positive integers with $0<m< n$.

We define the stack $\GG(d,m,n)$ as the stack whose objects consist of the data $(X\subset P\to S)$, where:
\begin{itemize}
    \item $P\to S$ is a Brauer-Severi variety with $n$-dimensional fibres, i.e. a scheme $P$ over $S$ such that there exists a covering $S'\to S$ in the \'etale topology with $P_{S'}\simeq \PP^n_{S'}$.
    \item $X\subset P$ is a closed subscheme and the induced morphism $X\to S$ is flat.
    \item For every geometric point $s$ in $S$, the fibre $X_s$ is a complete intersection of $m$ hypersurfaces of degree $d$ in $P_s\simeq \PP^n_{k(s)}$.
\end{itemize}
A morphism between two objects $(X\subset P\to S)$ and $(X'\subset P'\to S)$ is given by an isomorphism of Brauer-Severi varieties $P\simeq P'$ that induces an isomorphism $X\simeq X'$.

Let $W_d$ be the $\GL_{n+1}$-representation $\Sym^d E^{\vee} $, where $E$ is the standard $\GL_{n+1}$-representation. Let $\Gras$ be the grassmannian of $m$-dimensional subspaces in $W_d$: this variety has a well defined $\PGL_{n+1}$-action, because $\PGL_{n+1}$ naturally acts on $\PP(\wedge^m W_d)$ preserving the elements of the form $[f_1\wedge\dots\wedge f_m]$. 


\begin{prop}\label{prop:GG is a quotient}
     There exists a $\PGL_{n+1}$-invariant, closed subscheme $Z\subset\Gras$ of codimension $>1$ such that $\GG(d,m,n)\simeq [(\Gras\setminus Z)/\PGL_{n+1}]$. 
\end{prop}
\begin{proof}
    Consider the $\PGL_{n+1}$-torsor $\GG'(d,m,n)\to\GG(d,m,n)$ whose objects are pairs $(X\subset P\to S$, $\alpha)$, where $(X\subset P\to S)$ is an object of $\GG(d,m,n)$ and $\alpha:P\simeq\PP^n_S$ is an isomorphism over $S$.
    
    We claim that $\GG'(d,m,n)\simeq \Gras\setminus Z$ for some $\PGL_{n+1}$-invariant, closed subscheme $Z$ of the grassmannian of codimension $>1$.
    
    Observe that $\GG'(d,m,n)$ is equivalent to the stack whose objects are families of embedded complete intersections of equidegree $d$, i.e. closed subschemes $X\subset\PP^n_S$, flat over $S$, such that for each geometric point $s$ the fibre $X_s$ is a complete intersection of $m$ hypersurfaces of degree $d$ in $\PP^n_{k(s)}$.
    
    We can construct a morphism $\GG'(d,m,n)\to \Gras$ as follows: given an object $(X\subset\PP^n_S)$ of $\GG'(d,m,n)$, we have an exact sequence
    \[0\to \Ical_X(d)\to \Ocal_{\PP_S^n}(d)\to \Ocal_X(d)\to 0\]
    Pushing forward along the projection $\pr_1:\PP^n_S\to S$, we get:
    \[0\to \pr_{1*}\Ical_X(d)\to \pr_{1*}\Ocal(d)\to \pr_{1*}\Ocal_X(d)\to R^1\pr_{1*}\Ical_X(d)\]
    The sheaf $\pr_{1*}\Ocal(d)$ is isomorphic to $W_d\otimes\Ocal_S$.
    
    We want to prove that (1) $R^1\pr_{1*}\Ical_X(d)=0$, (2) the sheaf $\pr_{1*}\Ical_X(d)$ is locally free of rank $m$, and (3) the morphism from this sheaf to $W_d\otimes\Ocal_S$ is an embedding of locally free sheaves, i.e. injective on the fibres. Those three facts together will determine a morphism $\GG'(d,m,n)\to\Gras$.
    \begin{enumerate}
    \item Let $\Ical$ be the ideal of the closed embedding $X_s\subset\PP^n_{k(s)}$.
    By the cohomology and base change theorem (\cite{Hart}*{3.12.11}), it is enough show that $H^1(\PP^n_{k(s)},\Ical(d))=0$ for any geometric point $s$ in $S$.
    
    If $m=1$, then $\Ical(d)\simeq\Ocal$. If $m>1$, the fact that we that $X_s$ is a complete intersections implies that the Koszul complex
    \[ 0\to \bigwedge^m(\Ocal(-d)^{\oplus m})\to\cdots\to \bigwedge^2(\Ocal(-d)^{\oplus m})\to \Ocal(-d)^{\oplus m}\to \Ical\to 0\]
    is exact. Tensoring with $\Ocal(d)$ we obtain the exact sequence
    \[ 0\to \Ocal((1-m)d) \to \cdots\to\Ocal(-d)^{\oplus \binom{m}{2}}\to \Ocal^{\oplus m}\to \Ical(d)\to 0 \]
    It follows that $H^1(\PP^n_{k(s)},\Ical(d))=0$ if and only if $H^m(\PP^n_{k(s)},\Ocal((1-m)d))=0$, which is always the case as $m<n$.\\
    
    \item From (1) and the cohomology and base change theorem it follows that $\pr_{1*}\Ical(d)$ is locally free. Its rank is equal to the dimension of $H^0(\PP_{k(s)}^n,\Ical(d))$, which can be proved to be $m$ by a simple computation with the exact sequence above.\\
    
    \item From (1) and (2) we see that the restriction of the morphism $\pr_{1*}\Ical_X(d)\to\ W_d\otimes\Ocal_S$ over a geometric point $s$ of $S$ is equal to $H^0(\PP^n_{k(s)},\Ical_{X_s}(d))\to W_d$, which is obviously injective.\\
    
    \end{enumerate}
    
    The three points above determine a morphism $\GG'(d,m,n)\to \Gras$. 
    
    To produce the inverse morphism, consider the tautological locally free subsheaf $\Tcal\subset W_d\otimes\Ocal_{\Gras}$. 
    
    Observe that over the product variety $\Gras\times\PP^n$ we have the surjective morphism $W_d\otimes\pr_2^*\Ocal(-d)\to\Ocal$. 
    Define then $\Ical_X$ as the image along this morphism of $\pr_1^*\Tcal\otimes\pr_2^*\Ocal(-d)$, regarded as a subsheaf of $W_d\otimes\pr_2^*\Ocal(-d)$.

    Let $X$ be the subscheme of $\Gras\times\PP^n$ determined by the ideal sheaf $\Ical_X$: due to the fact that $\Tcal\subset W_{d}\otimes\Ocal_{\Gras}$ is an inclusion of locally free sheaves, we deduce that the fibres of $X$ over $\Gras$ are the zero locus of the image of the fibres of $\Tcal$ inside $W_d$, hence they are complete intersections of $m$ hypersurfaces of degree $d$. 
    
    Let $Z$ the closed subscheme of $\Gras$ whose complement is the flat locus of $X\to \Gras$: then we have a well defined morphism $\Gras\setminus Z \to \GG'(d,m,n)$.
    
    It is pretty much straightforward to check that the first of the two morphisms that we have constructed factors through $\Gras\setminus Z$, and that the two compositions are both equal to the identity.
    
    We only have to check that the codimension of $Z$ in $\Gras$ is $>1$. Let $D\subset \Gras$ be the closed subscheme where the morphism $X\to\Gras$ is not smooth: we claim that $D$ is an irreducible divisor.
    
    Observe that the general fibre over $D$ has only isolated singularities: in particular the generic element is flat over $D$, because it is a complete intersection of the right codimension.
    
    The subscheme $Z$ is contained in $D$: this follows from the fact that smoothness implies flatness. Therefore, $Z$ is a closed subscheme of an irreducible divisor, hence its codimension must be $>1$.
    
    We are only left with proving that $D$ is an irreducible divisor. Let $X^{\rm sing}$ be the singular locus of $X\to \Gras$: if we show that $X^{\rm sing}$ is irreducible and of dimension equal to dim$(\Gras)-1$, we are done.
    
    Consider the morphism $X^{\rm sing}\to \PP^n$: the irreducibility of $X^{\rm sing}$ is a consequence the irreducibility of the fibres over $\PP^n$.
    if we prove that the fibres of $X^{\rm sing}\to\PP^n$ are irreducible, this would imply the irreducibility of $X^{\rm sing}$.
    
    Actually, we only have to check that the fibre over the point $(0:\dots:0:1)$ of $\PP^n$ is irreducible, because if that is the case, then we can use the $\PGL_{n+1}$-action over $X^{\rm sing}$ to deduce the irreducibility of all the other fibres.
    
    The same argument also shows that the dimension of $X^{\rm sing}$ is equal to the dimension of $\Gras\times\PP^n$ minus the codimension of the fibre over $(0:\dots:0:1)$.
    
    Applying the Jacobian criterion of smoothness, we see that the fibre over $(0:\dots:0:1)$ of $X$ is isomorphic to the locus of $(n+1)\times m$-matrices of rank $<m$, and this is well known to be irreducible and of codimension $n+1$. From this the desired conclusion easily follows.
\end{proof}
\mypar Let $X$ be a scheme endowed with the action of a linear algebraic group $G$. We will denote $\Pic^G(X)$ the $G$-equivariant Picard group, i.e. the group of $G$-linearized line bundles.
There is an isomorphism
\[ \Pic([X/G])\simeq \Pic^G(X) \]
between the integral Picard group of the quotient stack $[X/G]$ and the $G$-equivariant Picard group \cite{EG}*{Proposition 18}.

\begin{prop}\label{prop:Pic GG}
    Let $\Tcal$ be the tautological bundle over $\Gras$, let $k:=mcm(n+1,md)/md$ and $p:=mcm(n+1,md)/(n+1)$ . Then: \[\Pic(\GG(d,m,n))\simeq\Pic^{\PGL_{n+1}}(\Gras)\simeq \ZZ\cdot[\det(\Tcal)^{\otimes k}] \]
    and the formula 
    \[((f_1(A^{-1}\underline{x}),\dots,f_m(A^{-1}\underline{x})),\det(A)^p( f_1(A^{-1}\underline{x})\wedge\dots\wedge f_m(A^{-1}\underline{x}))^{\otimes k})\]
    describes the $\PGL_{n+1}$-linearization of $\det(\Tcal)^{\otimes k}$.
\end{prop}
\begin{proof}
    By \propref{prop:GG is a quotient}, we know that $\Pic(\GG(d,m,n))\simeq\Pic^{\PGL_{n+1}}(\Gras)$, the group of isomorphism classes of $\PGL_{n+1}$-linearized line bundles over $\Gras$. The homomorphism from this group to $\Pic(\Gras)$ that forgets the $\PGL_{n+1}$-linearization is injective: this follows from the fact that $\PGL_{n+1}$ has no non-trivial character.
    
    The group $\Pic(\Gras)$ is freely generated by $\det(\Tcal)$, the determinant of the tautological bundle, hence $\Pic^{\PGL_{n+1}}(\Gras)$ is free as well. The points in the total space of $\det(\Tcal)$ are given by pairs $((f_1,\dots,f_m),(f_1\wedge\dots\wedge f_m))$, where the $f_i$ are linearly independent homogeneous forms of degree $d$. We have to find the minimum $k>0$ such that $\det(\Tcal)^{\otimes k}$ admits a $\PGL_{n+1}$-linearization, or equivalently a $\GG_m$-invariant $\GL_{n+1}$-linearization.
    
    Any $\GL_{n+1}$-linearization must be of the following form: given an element $A$ of $\GL_{n+1}$, it acts on the point above by sending 
    \[(f_1,\dots,f_m)\longmapsto (f_1(A^{-1},\underline{x})\dots,f_m(A^{-1}\underline{x}))\] and
    \[ A\cdot (f_1\wedge\dots\wedge f_m)^{\otimes k}= \det(A)^p (f_1(A^{-1}\underline{x})\wedge \dots \wedge f_m(A^{-1}\underline{x}))^{\otimes k}   \]
    The subtorus $\GG_m\subset \GL_{n+1}$ of scalar matrices acts as
    \[ \lambda \cdot  (f_1\wedge\dots\wedge f_m)^{\otimes k}= \lambda^{p(n+1)}\cdot \lambda^{-mdk}\cdot  (f_1\wedge\dots\wedge f_m)^{\otimes k} \]
    The minimum $k$ such that the action above is trivial can be read off from the minimum common multiple of $n+1$ and $md$: more precisely, it is equal to $mcm(n+1,md)/md$, and $p=mcm(n+1,md)/(n+1)$. This concludes the proof of the Proposition.
\end{proof}
\mypar \label{par:hodge for GG} Given a family of complete intersections $(X\subset P\to S)$ in $\GG(d,m,n)$, let $\pi:X\to S$ be the induced (flat) morphism.

We can define a sheaf on $S$ as follows:
\[ \Phi(S):= H^0(S,\omega_{X/S}^{\otimes k}\otimes(\omega_{P/S}|_X)^{\otimes (p-k)})\]
where $p$ and $k$ are as in \propref{prop:Pic GG}, and $\omega_{X/S}$ and $\omega_{P/S}$ are the relative dualizing sheaves. This defines a sheaf
\[ \Phi: \GG(d,m,n)_{{\rm smooth}}^{{\rm op}} \longrightarrow ({\rm Ab})\]
on the small smooth site of $\GG(d,m,n)$.

\begin{cor}\label{cor:Pic GG 2}
    Let $k,p$ be as in \propref{prop:Pic GG} and let $\Phi$ be the sheaf defined in \parref{par:hodge for GG}. Then $\Phi$ is a line bundle and it freely generates $\Pic(\GG(d,m,n))$.
\end{cor}
\begin{proof}
    Let $\Gras^o:=\Gras\setminus Z$, the open subscheme of $\Gras$ that appears in \propref{prop:GG is a quotient}.
    
    It is enough to show that the pullback of $\Phi$ along the $\PGL_{n+1}$-torsor \[f: \Gras\setminus Z \rightarrow\GG(d,m,n)\] is a line bundle, and that it generates $\Pic^{\PGL_{n+1}}(\Gras)$. We have:
    \[ f^*\Phi\simeq \pi_*\left(\omega_{X/(\Gras^o)}^{\otimes k}\otimes (\pr_2^*\omega_{\PP^n}|_{X})^{\otimes (p-k)} \right) \]
    where $X\subset\Gras^o\times\PP^n$ is the family of complete intersections constructed in the proof of \propref{prop:GG is a quotient}.
    
    The normal bundle $N$ of $X$ is equal to the restriction to $X$ of $\pr_1^*\Tcal^{\vee}\otimes\pr_2^*\Ocal(d)$, hence
    \[ \det(N)=\pr_1^*\det(\Tcal)^{\vee} \otimes \pr_2^*\Ocal(dm). \]
    The dualizing sheaf of the local complete intersection morphism $\pi:X\to \Gras^o$ is isomorphic to
    \begin{align*} 
    \det(N)\otimes (\omega_{(\Gras^o\times\PP^n)/\Gras^o}|_X) &= \det(N)\otimes(\pr_2^*\omega_{\PP^n})|_X \\
    &=\pr_1^*\det(\Tcal)^{\vee} \otimes (\pr_2^*\omega_{\PP^n}(dm))|_X
    \end{align*}
    Therefore we get:
    \[ f^*\Phi\simeq \pr_{1*}\left(\pr_1^*\det(\Tcal^{\vee})^{\otimes k}\otimes\pr_2^*\omega_{\PP^n}^{\otimes p}(kmd)\right )\]
    Applying the projection formula, and using the relation $kmd=p(n+1)$, we conclude that $f^*\Phi$ is isomorphic to $\det(\Tcal^{\vee})$, a generator of $\Pic^{\PGL_{n+1}}$.
\end{proof}

\mypar \label{par:GGrat} There are two substacks of $\GG(d,m,n)$ that will be of some interest for us:
\begin{itemize}
    \item $\GG(d,m,n)\rat$, the stack of complete intersection with at most ADE singularities.
    
    \item $\GG(d,m,n)\sm$, the stack of smooth complete intersections.
\end{itemize}

\subsection{Complete intersections of codimension 2 and bidegree $(a,b)$}

\mypar \label{par:FF} Fix three integers $a$, $b$, $n$ with $n>0$ and $b>a>0$. We define the stack $\FF(a,b,n)$ as the stack whose objects are determined by the data $(X\subset P\to S)$, where:
\begin{itemize}
    \item $P\to S$ is a Brauer-Severi variety with $n$-dimensional fibres.
    \item $X\subset P$ is a closed embedding and the induced morphism $X\to S$ is flat.
    \item For every geometric point $s$ in $S$, the fibre $X_s$ is a complete intersection of two hypersurfaces of degree $a$ and $b$.
\end{itemize}
A morphism between two objects $(X\subset P\to S)$ and $(X'\subset P'\to S)$ is given by an isomorphism of Brauer-Severi varieties $P\simeq P'$ that induces an isomorphism $X\simeq X'$.

\mypar \label{par:P(Vab)} Let $E$ be the standard $\GL_{n+1}$-representation and set $W_d:=\Sym^d E^{\vee}$ for $d>0$.

The morphism of representations $W_a\otimes W_{b-a}\to W_b$ induces an injective morphism of vector bundles over $\PP(W_a)$, that is:
\[\varphi: W_{b-a}\otimes\Ocal_{\PP(W_a)}(-1)\to W_{b}\otimes\Ocal_{\PP(W_a)} \]
Let $\PP(V_{a,b})\to\PP(W_a)$ be the projectivization of the quotient vector bundle
\[ V_{a,b}:=\frac{W_b\otimes\Ocal_{\PP(W_a)}}{\im(\varphi)}\]
Observe that both $\PP(V_{a,b})$ and $\PP(W_a)$ inherit a $\PGL_{n+1}$-action.

\begin{prop}\label{prop:FF is a quotient}
    There exists a $\PGL_{n+1}$-invariant closed subscheme $Z\subset\PP(V_{a,b})$ of codimension $>1$ such that $ \FF(a,b,n)\simeq [(\PP(V_{a,b})\setminus Z)/\PGL_{n+1}]$.
\end{prop}
\begin{proof}
    The proof is similar to the one of \propref{prop:GG is a quotient}, hence we will only give here a sketch of the key steps, without diving too much into the details.
    
    Consider the $\PGL_{n+1}$-torsor $\FF'(a,b,n)\to \FF(a,b,n)$ whose objects are of the form $(X\subset P\to S,\alpha)$, where $(X\subset P\to S)$ is an object of $\FF(a,b,n)$ and $\alpha:P\simeq\PP^{n}_S$ is an isomorphism: we want to show that $\FF'(a,b,n)$ is isomorphic to $\PP(V_{a,b})\setminus Z$, where $Z$ is a $\PGL_{n+1}$-invariant, closed subscheme of codimension $>1$.
    
    First we construct the morphism $\FF'(a,b,n)\to\PP(V_{a,b})$ using the characterization of the functor of points of these schemes.
    
    Take an object $(X\subset\PP^{n}_S\to S)$ of $\FF'(a,b,n)$. A routine application of cohomology and base change theorem shows that the following is an exact sequence of locally free sheaves:
    \[ 0\to \pr_{2*}(\Ical_X\otimes\pr_1^*\Ocal(a))\to \to \pr_{2*}\pr_1^*\Ocal(a)\to \pr_{2*}(\Ocal_X\otimes\pr_1^*\Ocal(a))\to 0 \]
    In particular, the sheaf in the middle is isomorphic to $W_a\otimes\Ocal_S$ and the one on the left has rank $1$. This induces a morphism $f:S\to \PP(W_a)$ such that $f^*\Ocal(-1)\simeq \pr_{1*}(\Ical_X\otimes\pr_2^*\Ocal(b))$.
    
    We also get the following diagram:
    \[\xymatrix{
    \pr_{1*}(\Ical_X\otimes\pr_2^*\Ocal(a))\otimes W_{b-a} \ar[r]^{\simeq} \ar[d] & f^*\Ocal_{\PP(W_a)}(-1)\otimes W_{b-a} \ar[d] \\
    \pr_{1*}(\Ical_X\otimes\pr_2^*\Ocal(b)) \ar[r] & W_b\otimes\Ocal_S }\]
    Define $\Lcal$ as the quotient of $\pr_{1*}(\Ical_X\otimes\pr_2^*\Ocal(b))$ by $\pr_{1*}(\Ical_X\otimes\pr_2^*\Ocal(a))\otimes W_{b-a}$: this can be proved to be a line bundle, and from the commutativity of the diagram above we deduce that it is a line subbundle of $f^*V_{a,b}$.
    
    Therefore, the line bundle $\Lcal$ determines a morphism $S\to\PP(V_{a,b})$.
    
    The whole construction above is functorially well-behaved, in the sense that it induces a natural transformation between the functors of points of $\FF'(a,b,n)$ and $\PP(V_{a,b})$, and hence it determines a morphism $\FF'(a,b,n)\to\PP(V_{a,b})$.
    
    To produce a morphism in the other direction, we have to construct a universal family of complete intersections in $\PP^n$ over $\PP(V_{a,b})$.
    
    Let $Y\subset \PP(W_a)\times\PP^n$ be the universal hypersurface over $\PP(W_a)$, and consider its pullback $Y':=\pi^{-1}Y$ along $\pi:\PP(V_{a,b})\to\PP(W_a)$.
    
    By construction $Y'$ lives in $\PP(V_{a,b})\times\PP^n$. On this product scheme we have the following injective morphism of sheaves:
    \[ \pr_1^*\Ocal_{\PP(V_{a,b})}(-1)\otimes\pr_2^*\Ocal_{\PP^n}(-b)\hookrightarrow \pr_1^*V_{a,b}\otimes\pr_2^* \Ocal_{\PP^n}(-b) \]
    Observe that the restriction of $\pr_2^*V_{a,b}$ to $Y'$ is equal to $W_b\otimes\Ocal_{Y'}$. In this way we can define the following morphism:
    \[ \pr_1^*\Ocal_{\PP(V_{a,b})}(-1)\otimes\pr_2^*\Ocal_{\PP^n}(-b)|_{Y'}\to W_b\otimes\pr_2^*\Ocal(-n)|_{Y'}\to \Ocal_{Y'} \]
    The image of this morphism is an ideal inside $Y'$: let $X$ be the associated closed subscheme of $Y'$.
    
    The fibres of $X$ over $\PP(W_{a,b})$ are complete intersections inside $\PP^n$ of two hypersurfaces of degree $a$ and $b$. 
    
    Let $Z\subset \PP(V_{a,b})$ be the locus where $X\to\PP(V_{a,b})$ is non-flat. Then $Z$ is strictly contained in the singular locus of $X\to\PP(V_{a,b})$, which can be proved to be an irreducible divisor using the same arguments of the proof of \propref{prop:GG is a quotient}.
    
    We deduce that $Z$ has codimension $>1$.
    
    This construction determines a morphism $\PP(V_{a,b})\setminus Z\to \FF'(a,b,n)$. It is easy to check that the compositions of the two morphisms we have introduced so far are both equal to the identity, hence $\PP(V_{a,b})\setminus Z\simeq \FF'(a,b,n)$, which concludes the proof.
\end{proof}
    
\begin{prop}\label{prop:Pic FF}
    The group $\Pic(\FF'(a,b,n))$ is isomorphic to the sublattice 
    \[ \langle [\Ocal_{\PP(V_{a,b})}(-d)\otimes\pi^*\Ocal_{\PP(W_a)}(-l)] \rangle\subset \Pic(\PP(V_{a,b})\simeq \ZZ^{\oplus 2}\]
    where $d$ and $l$ are such that $n+1$ divides $ad+bl$.
\end{prop}
\begin{proof}
    \propref{prop:FF is a quotient} implies that $\Pic(\FF(a,b,n))$ is isomorphic to the sublattice $\Pic^{\PGL_{n+1}}(\PP(V_{a,b}))$ of $\Pic(\PP(V_{a,b}))$.
    
    The latter is freely generated by $[\Ocal_{\PP(V_{a,b})}(-1)]$ and $[\pi^*\Ocal_{\PP(W_a)}(-1)]$. We have to find those elements that admits a $\PGL_{n+1}$-linearization.
    
    Any possible $\PGL_{n+1}$-linearization on $\Ocal_{\PP(V_{a,b})}(-d)\otimes\pi^*\Ocal_{\PP(W_a)}(-l)$ must be of the form:
    \[ A\cdot f(\underline{x})^{\otimes d}\otimes g(\underline{x})^{\otimes l}:=\det(A)^{p} f(A^{-1}\underline{x})^{\otimes d}\otimes g(A^{-1}\underline{x})^{\otimes l}\]
    The condition under which the scalars act trivially is that $(n+1)p=al+bd$. From this the Proposition follows.
\end{proof}

\mypar \label{par:hodge for FF} Observe that given a family of complete intersections $(\pi:X\subset P\to S)$ in $\FF(a,b,n)$, there exists a unique hypersurface $Y\subset P$ of degree $a$ that contains $X$.

We can define the sheaf:
\[ \Phi_{l,d}(S):=\pi_*\left(\omega_{X/S}^{\otimes d}\otimes(\omega_{Y/S}|_X)^{\otimes{l-d}}\otimes (\omega_{P/S}|_X)^{\otimes {p-l+d}}\right)\]

It is pretty straightforward to check that there exists a globally defined sheaf $\Phi_{l,d}$ whose pullback along $S\to\FF(a,b,n)$ is equal to $\Phi(S)$.

\begin{cor}\label{cor:Pic FF 2}
    Let $d,l,p$ be such that $ad+bl=(n+1)p$. The the sheaves $\Phi_{l,d}$ defined in \parref{par:hodge for FF} are line bundles. Morever, the Picard group of $\FF(a,b,n)$ is free of rank $2$ and it is formed by the line bundles $\Phi_{l,d}$ defined in \parref{par:hodge for FF}.
\end{cor}
\begin{proof}
    Set $\PP(V_{a,b})^o:=\PP(V_{a,b})\setminus Z$, where the latter is the open subscheme of $\PP(V_{a,b})$ introduced in \propref{prop:FF is a quotient}.
    
    We will argue as in the proof of \corref{cor:Pic GG 2}.
    Consider the $\PGL_{n+1}$-torsor $f:\PP(V_{a,b})\setminus Z\to\FF(a,b,n)$ given by \propref{prop:FF is a quotient}.
    
    We want to show that $f^*\Phi_{l,d}\simeq \Ocal_{\PP(V_{a,b})}(l)\otimes \pi^*\Ocal_{\PP(W_a)}(d)$. This would allows us to conclude by \propref{prop:Pic FF}.
    
    Let $X\subset \PP(V_{a,b})^o\times\PP^n$ be the complete intersection constructed in the proof of \propref{prop:FF is a quotient}, and let $Y$ be the pullback from $\PP(W_a)\times\PP^n$ of the universal degree $a$ hypersurface over $\PP(W_a)$.
    
    We have:
    \begin{align*}
        &\omega_{X/\PP(V_{a,b})^o}\simeq \det(N_{X/Y})\otimes \omega_{Y/\PP(V_{a,b}^o)}|_X \\
        &\omega_{Y/\PP(V_{a,b})^o}\simeq \det(N_Y)\otimes \pr_2^*\omega_{\PP^n}|_Y
    \end{align*}
    The normal bundle $N_{X/Y}$ of $X$ in $Y$ is isomorphic to \[\left(\pr_1^*\Ocal_{\PP(V_{a,b})}(1)\otimes\pr_2^*\Ocal_{\PP^n}(b)\right)|_X \]
    The normal bundle $N_Y$ of $Y$ in $\PP(V_{a,b})^o\times\PP^n$ is isomorphic to \[\pr_1^*\pi^*\Ocal_{\PP(W_a)}(1)\otimes\pr_2^*\Ocal_{\PP^n}(a)\]
    where $\pi:\PP(V_{a,b})^o\to\PP(W_a)$ is the projection morphism.
    
    A straightforward computation shows the following:
    \[ f^*\Phi_{l,d}\simeq \pr_{1*}\left(\pr_1^*(\Ocal_{\PP(V_{a,b})^o}(l)\otimes \pi^*\Ocal_{\PP(W_a)}(d) )\otimes \pr_2^*( \omega_{\PP^n}^{\otimes p}(ad+bl))\right) \] 
    Applying the projection formula, we get the desired conclusion.
\end{proof}

\mypar \label{par:FFrat} There are two open substacks of $\FF(a,b,n)$ which will play a role in the remainder of the paper, namely:
\begin{itemize}
    \item The stack $\FF(a,b,n)\rat$ of complete intersections with at most ADE singularities.
    \item The stack $\FF(a,b,n)\sm$ of smooth complete intersections.
\end{itemize}
\section{Recap on moduli of polarized K3 surfaces with at most rational double points}\label{sec:recap}
In this section we briefly recall some notions from the theory of moduli of polarized K3 surfaces with at most rational double points. A standard reference for this is \cite{Huy}*{Chapter 5}. In this section, every scheme is assumed to be defined over $\mathbb{Q}$.

\mypar \label{par:def quasi-pol} Let $X$ be a K3 surface over a field $k$ with at most rational double points, i.e. a geometrically connected, integral and proper scheme of dimension $2$ over $\Spec(k)$ whose singularities are at most rational double points, and such that $\omega_{X/k}\simeq\Ocal_X$ and $H^1(X,\Ocal_X)=0$. 

The Picard functor $\Pic_{X/k}$ is represented by a separated, smooth $0$-dimensional scheme over $k$: we define a \emph{polarization of degree} $l$ for $X$ as a rational point $\sigma$ of $\Pic_{X/k}$ that on the algebraic closure $\overline{k}$ of $k$ is the class of an ample line bundle $L$ such that $(L^2)=l$, where the degree of a line bundle is defined as the self-intersection number of its first Chern class.
A polarization $\sigma$ is primitive if the line bundle $L$ is not divisible in $\Pic_{X_{\overline{k}}/\overline{k}}$.

\mypar A \emph{family of K3 surfaces with at most rational double points} is a proper and flat morphism of schemes $\pi:X\to S$ whose fibres are K3 surfaces with at most rational double points.

As before, the relative Picard functor $\Pic_{X/S}$ is represented by a separated algebraic space, locally of finite presentation: we define a polarization of degree $l$ of the family $\pi:X\to S$ as a section $\sigma$ of $\Pic_{X/S}\to S$ such that for every geometric point $\overline{s}$ of $S$ the restriction $\sigma_{\overline{s}}$ is a polarization of degree $l$ for $X_{\overline{s}}\to \overline{s}$. 

A polarization for a family of K3 surfaces is primitive if its restriction to every geometric fibre is primitive.

An isomorphism of polarized K3 surfaces with at most rational double points $(X,\sigma)\to (X',\sigma')$ is an isomorphism $f:X\to X'$ such that $f^*\sigma'=\sigma$. Similarly, an isomorphism of families of polarized K3 surfaces with at most rational double points $(X/S,\sigma)\to (X'/S',\sigma')$ is given by a morphism $f:S\to S'$ together with an isomorphism $g:X\to X'\times_{S'} S$ such that t$g^*\sigma'=\sigma$.

\mypar Let $2l\geq 4$ be a positive number and define $\Mcal_{2l}$ as the fibred category over the site of schemes $(Sch/\mathbb{Q})_{fppf}$ whose objects are families of K3 surfaces with at most rational double points together with a primitive polarization of degree $2l$. 

The morphisms in $\Mcal_{2l}$ are isomorphisms of polarized K3 surfaces. By \cite{Huy}*{pg. 84}(note that in \emph{loc. cit.} the stack $\Mcal_{2l}$ is denoted $\bar{\Mcal}_d$) we have that $\Mcal_{2l}$ is a smooth Deligne-Mumford stack that admits a coarse moduli space $M_{2l}$: this scheme coincides with the usual moduli space of polarized K3 surfaces with at most rational double points constructed through the period map.


\mypar The integral Picard group $\Pic(\Mcal_{2l})$ can be defined either as the group of invertible sheaves over $\Mcal_{2l}$ or as the first cohomology group $H^1(\Mcal_{2l},\Ocal^*)$: these two approaches are in fact equivalent (see \cite{Mum63}*{Sec. 5}). 

Moreover, due to the fact that $\Mcal_{2l}$ is smooth, there is an isomorphism between the divisor class group of $\Mcal_{2l}$ and its Picard group: in particular,  every closed substack $\Dcal\subset\Mcal_{2l}$ of pure codimension $1$ induces an exact sequence:
\[ \oplus_i \ZZ\cdot [\Dcal_i] \to \Pic(\Mcal_{2l}) \to \Pic(\Mcal_{2l}\setminus \Dcal) \to 0 \]
where the $\Dcal_i$ are the irreducible components of $\Dcal$.

\mypar \label{par:NL} We recall here the definition of the Noether-Lefschetz divisors inside $\Mcal_{2l}$: for a more detailed account of this theory, see \cite{MP}*{Sec. 1}. 
For $d>0$, the substack $\Dcal_{d,h}$ is defined as the stack of primitively polarized K3 surfaces with at most rational double points $(X\to S,\sigma)$ of degree $2l$ such that there exists an element $\beta$ in $\Pic_{X/S}(S)$ whose restriction over any geometric point $\overline{s}$ of $S$ satisfies:
\[ (\beta\cdot\beta)_{X_{\overline{s}}}=2h-2,\quad (\beta\cdot\sigma)_{X_{\overline{s}}}=d. \]
We also define $\Dcal_{0,0}$ as the substack of K3 surfaces that have at least a singular point.
If $d^2+2l(1-h)>0$ then $\Dcal_{d,h}\subset \Mcal_{2l}$ is an irreducible divisor. The divisors defined in this way are called \emph{Noether-Lefschetz divisors}.

\mypar \label{par:hodge line bundle} Consider the sheaf on $\Mcal_{2l}$ defined as follows:
\[ (\pi:X\to S,\sigma)\longmapsto \pi_*\omega_{X/S} \]
This sheaf is actually a line bundle.

Indeed, observe that the restriction of $\omega_{X/S}$ to any fibre $\pi^{-1}(s)=:X_s$ is trivial.
Applying the cohomology and base change theorem (\cite{Hart}*{Thm. 3.12.11}) we immediately deduce that $R^1\pi_*\omega_{X/S}=0$ and that $\pi_*\omega_{X/S}$ is locally free of rank equal to $h^0(X_s\Ocal_{X_s})$, i.e. is a line bundle.

We will refer to this line bundle as the \emph{Hodge line bundle}, and its class in $\Pic(\Mcal_{2l})$ will be denoted $\lambda_1$.

\begin{lm}\label{lm:hodge for K3}
     Let $(\pi:X\to S,\sigma)$ be a family of primitively polarized K3 surfaces with at most rational double points of degree $2l$.
     Up to passing to an \'{e}tale cover of $X$, we can assume that there is a line bundle $L$ such that $\sigma=[L]$ in $\Pic(X)/\pi^*\Pic(S)$.
     Then $\pi_*(L^{\otimes p})$ is a locally free sheaf on $S$ of rank $p^2l+2$. 
\end{lm}
\begin{proof}
    We will show that $\pi_*(L^{\otimes p})$ is locally free using the cohomology and base change theorem (\cite{Hart}*{Thm. 3.12.11}). All we need to do is to prove that $H^1(X_{\overline{s}},L_{\overline{s}}^{\otimes p})$ is equal to $0$ for every geometric point of $S$.
    
    This can be done using the Kawamata-Viehweg vanishing theorem, as $L^{\otimes p}$ is ample.
    
    We know by hypothesis that $(L_{\overline{s}}^2 )=2l$ and moreover $h^1(L_{\overline{s}}^{\otimes p})=h^2(L_{\overline{s}}^{\otimes p})=0$ by Kawamata-Viehweg vanishing theorem. Riemann-Roch formula then tells us that $h^0(L_{\overline{s}})=2+p^2l$.
\end{proof}

\mypar \label{par:psiL} Let $(\pi:X\to S,\sigma)$ be a family of primitively polarized K3 surfaces with at most rational double points. \lmref{lm:hodge for K3} gives us a line bundle $L$ on $X$ and the canonical morphism $\pi^*\pi_*L\to L$ induces a rational morphism $X\dashrightarrow \PP((\pi_*L)^{\vee})$ of $S$-schemes.

Being $X$ normal and $\PP(\pi_*L^\vee)$ proper over $S$, the rational morphism extends in a unique way to the complement of a codimension $2$ subscheme of $X$. By \cite{SD}*{Cor. 3.2}, there exists a unique extension $\psi_\sigma:X\to\PP(\pi_*L^\vee)$ of the rational morphism to the whole family. 

Observe that  the morphism $\psi_\sigma$ is uniquely determined by $\sigma$, although the line bundle $L$ is not.

\begin{prop}\label{prop:proj model}
     Let $(X,\sigma)$ be a primitively polarized K3 surface with at most rational double points of degree $2l\geq 4$, $L$ a line bundle on $X$ representing $\sigma$ and $\psi_\sigma:X\to\PP( H^0(X,L)^\vee)$ the induced morphism described in \parref{par:psiL}. Then one of the following occurs:
     \begin{enumerate}
         \item $\psi_\sigma$ is birational onto its image.
         \item $\psi_\sigma$ is finite of degree $2$: this happens if and only if $(X,\sigma)$ is a point of $\Dcal_{2,1}$ in $\Mcal_{2l}$.
         \item The image of $\psi_\sigma$ is a curve: this happens if and only if $(X,\sigma)$ is a point of $\Dcal_{1,1}$ in $\Mcal_{2l}$.
     \end{enumerate}
\end{prop}

\begin{proof}
    Consider the resolution $\rho:X'\to X$ together with the pullback of $\sigma$: this pair is a primitively quasi-polarized K3 surface, meaning that $\sigma'=\rho^*\sigma$ is nef and big, and for every curve $C$ such that $\sigma\cdot C=0$, we have $(C)^2=(-2)$. We will work with $(X',\sigma')$ rather than $X$ in what follows, so to leverage some results of Saint-Donat.
    
    We start by proving $(3)$.
    
    Suppose that the image of $\psi_{\sigma'}:X'\to\PP(H^0(X,L)^\vee)$ is a curve $C$. Denote $H$ a hyperplane section of $\PP(H^0(X,L)^\vee)$: we can assume that $H\cap C=p_1+\dots +p_k$, where the $p_i$ are all regular and distinct. Then $L=\psi_{\sigma'}^*(H\cap X)=D+(E_1+\dots +E_k)$, where $D$ is the fixed component of $L$ and each $E_i$ has genus $1$ (this follows from the adjunction formula).
    
    Observe that $(E_i^2)=0$ and $(E_i\cdot E_j)=0$, which implies that $D\neq 0$. Moreover, by \cite{SD}*{Prop.2.6} the divisor $E_1+\dots +E_k$ is linearly equivalent to $kE$ for some genus $1$, irreducible curve $E$. By \cite{SD}*{(2.7.4)} we have $D=D_1+\dots +D_n$ with $(D_1\cdot E)=1$ and $(D_i\cdot E)\leq 0$ for $i>1$. By hypothesis, as $(E^2)\neq (-2)$, we have:
    \[ 0<(L\cdot E)=k(E^2)+(D_1\cdot E)+\sum_{i>1} (D_i\cdot E)=1+\sum_{i>1} (D_i\cdot E) \]
    The inequality $(D_i\cdot E)\leq 0$ for $i>1$ implies that $(D_i\cdot E)=0$ and $(L\cdot E)=1$, hence $(X,\sigma)$ is in $\Dcal_{1,1}$.
    
    Suppose now that there exists $E\subset X'$ such that $(E^2)=0$ and $(E\cdot L)=1$. Then $(L-(l+1)E)^2=(-2)$ and $(L\cdot (L-(l+1)E))=l-2$. For $2l\geq 4$, this implies that there exists an effective divisor $D$ with $(D^2)=-2$ such that $L=(l+1)E+D$. Observe that $h^0(\Ocal(E))=2$, hence $h^0(\Ocal((l+1)E))\geq l+2$. By hypothesis $h^0(L)=l+2$, thus $h^0(\Ocal((l+1)E))=l+2$ and $D$ is a fixed component.
    
    This implies that $\psi_{\sigma'}$ is actually the morphism $\psi_{(l+1)E}$ induced by the base-point free linear system $|(l+1)E|$. Moreover, every divisor in $|E|$, which has dimension $1$, get contracted by $\psi_{(l+1)E}$, so that the image of this morphism must necessarily be a curve.
    
    Point $(2)$ is \cite{SD}*{Thm. 5.2}, and point $(1)$ is \cite{SD}*{(4.1)}.
\end{proof}


\section{Integral Picard group of $\Mcal_{2l}$ for $l=2,3,4$}\label{sec:Pic}

In this section we compute the Picard group of $\Mcal_{2l}$ for $l=2,3,4$  (see \thmref{thm:Pic K4}, \thmref{thm:Pic(K6)} and \thmref{thm:Pic K8}).

These Picard groups turn out to be free abelian groups on respectively $3$, $4$ and $4$ generators. Moreover, the generators of these Picard groups are in each case given by the Hodge line bundle (see \parref{par:hodge line bundle}) and some (elliptic) Noether-Lefschetz divisors (see \parref{par:NL}).

From now on, every polarization is assumed to be primitive.

Moreover, we will always be assuming that there exists a line bundle representing the polarization.

This may not be true on the nose, but only up to passing to an \'{e}tale cover. Nevertheless, this technical detail will not be playing any role in our arguments, so its omission should be harmless.

\subsection{Computation of $\Pic(\Mcal_4)$}

\mypar Let $(X\to S,\sigma)$ be a family of polarized K3 surfaces with at most rational double points of degree $4$. \propref{prop:proj model} tells us that the polarization $\sigma=[L]$ induces a morphism $\psi_\sigma:X\to\PP(\pi_*L^\vee)$, which can be either birational, finite of degree $2$ or of relative dimension $1$.

Moreover, the substack of families whose polarization induces a birational morphism is an open substack of $\Mcal_4$, namely the complement of the Noether-Lefschetz divisors $\Dcal_{1,1}$ and $\Dcal_{2,1}$. We will refer to this open substack as $\Ucal_4$.

Recall that in \parref{par:GGrat} we introduced the stack $\GG(d,m,n)\rat$ of complete intersections of equidegree $d$ and codimension $m$ inside a Brauer-Severi variety with $n$-dimensional fibres, having at most ADE singularities.

\begin{prop}\label{prop:U4 iso}
 The stack $\Ucal_4$ is isomorphic to $\GG(4,1,3)\rat$.
\end{prop}

\begin{proof}
    Given a family of polarized K3 surfaces with at most rational double points $(\pi:X\to S,\sigma)$ of degree $4$ whose polarization induces a birational morphism $\psi_\sigma:X\to \PP(\pi_*L^\vee)$, the image $\overline{X}:=\psi_\sigma(X)$ will be a degree $4$ hypersurface in $\PP(\pi_*L^\vee)$ whose singularities are at most rational double points. 
    
    An automorphism of the family of polarized K3 surfaces will induce an automorphism of $\PP(\pi_*L^\vee)$, hence of $\overline{X}$.
    
    Therefore, we have a well defined morphism of stacks $f:\Ucal_4\to \GG(4,1,3)\rat$, which sends a family of polarized K3 surfaces with at most rational double points $(X\to S,\sigma)$ of degree $4$ to its projective model $(\overline{X}\subset \PP(\pi_*L^\vee))$. We claim that this is actually an isomorphism of stacks.

    For showing this, we construct explicitly the inverse map: given an object $(X\subset P\to S)$ of $\GG(4,1,3)\rat(S)$, up to passing to an \'etale cover of $S$, we can suppose that $P\simeq \PP(E)$ for some vector bundle $E\to S$. In particular, there is a well defined line bundle $\Ocal(1)$ on $\PP(E)$, which we can restrict to $X$. Then the pair $(X\to S,\Ocal(1)|_{X})$ is a polarized K3 surface with at most rational double points of degree $4$, and such that the polarization is very ample.

    This defines a morphism $\GG(4,1,3)\rat \to \Ucal_4$ which is the inverse of $f$, thus concluding the proof.

\end{proof}

\begin{cor}\label{cor:U4 is a quot}
    The stack $\GG(4,1,3)\rat$ is smooth, open in $\GG(4,1,3)$ and its complement has codimension $>1$.
\end{cor}
\begin{proof}
    By \propref{prop:U4 iso} we know that $\GG(4,1,3)\rat\simeq\Ucal_4$, and this is an open substack of the Deligne-Mumford stack $\Mcal_4$, which is smooth (\cite{Huy}*{pg. 84}). 
    
    This readily implies that $\GG(4,1,3)\rat$ is smooth inside $\GG(4,1,3)$, and it has the same dimension, hence it is open. 
    
    The complement of $\GG(4,1,3)\rat$ is strictly contained in the divisor of singular complete intersections, which is irreducible, as discussed in the end of the proof of \propref{prop:GG is a quotient}): therefore, the complement of $\GG(4,1,3)\rat$ must have codimension $>1$.
\end{proof}

\mypar Let $E$ be the standard $\GL_4$-representation, and define $W_4:=\Sym^4 E^{\vee}$.

\label{par:lambda4} Consider the line bundle $\Ocal(-1)$ over $\PP(W_4)$ with the $\PGL_4$-linearization given by the formula
\[ A\cdot f(\underline{x}):= \det(A)f(A^{-1}\underline{x}) \]
By \propref{prop:Pic GG}, the $\PGL_4$-quotient of $\Ocal(-1)$ defines a line bundle over the stack $\GG(4,1,3)$ which freely generate its Picard group. 

\corref{cor:U4 is a quot} tells us that the codimension of the complement of $\GG(4,1,3)\rat$ in $\GG(4,1,3)$ is $>1$, hence the Picard groups of the two stacks are isomorphic. By \propref{prop:U4 iso} we know that $\Ucal_4\simeq \GG(4,1,3)\rat$, and then we deduce that $\Pic(\Ucal_4)$ is freely generated over one element.

From \corref{cor:Pic GG 2} we deduce that the generator of $\Pic(\Ucal_4)$ is the restricted Hodge line bundle $\lambda_1|_{\Ucal_4}$.

We are ready to prove the main result of the subsection.

\begin{thm}\label{thm:Pic K4}
    We have:
    \begin{enumerate}
        \item $\Pic(\Ucal_4)\simeq \ZZ\cdot \lambda_1|_{\Ucal_4}$.
        \item $\Pic(\Mcal_4)\simeq \ZZ\cdot [\Dcal_{1,1}]\oplus \ZZ\cdot [\Dcal_{2,1}] \oplus \ZZ\cdot\lambda_1$.
    \end{enumerate}
\end{thm}
\begin{proof}
    We already proved (1) in \parref{par:lambda4}. Consider the exact sequence
    \[ \ZZ\cdot[\Dcal_{1,1}]\oplus\ZZ\cdot[\Dcal_{2,1}]\to\Pic(\Mcal_4)\to\ZZ\cdot(\lambda_1|_{\Ucal_4})\to 0 \]
    We know from \cite{Bru}*{(6)} that $\Pic_{\mathbb{Q}}(\Mcal_4)$ has rank $\geq 3$, hence the arrow on the left in the sequence above must be injective. This implies (2).
\end{proof}

\subsection{Degree six case}

\mypar Let $\Ucal_6$ be the open substack of $\Mcal_6$ formed by those families $(\pi:X\to S,\sigma)$ of polarized K3 surfaces of degree $6$ with at most rational double points such that the polarization induces a birational morphism, i.e. the image of $\psi_{\sigma}:X\to P$ is birational to $X$.

Recall that in \parref{par:FFrat} we defined the stack $\FF(a,b,n)\rat$ whose objects are families of complete intersections of codimension $2$ and bidegree $(a,b)$ in a Brauer-Severi variety having $n$-dimensional fibres such that the fibres have at most ADE singularities.

\begin{prop}\label{prop:U6 iso}
    The stack $\Ucal_6$ is isomorphic to $\FF(2,3,4)\rat$.
\end{prop}

\begin{proof}
    The proof is similar to the one of \propref{prop:U4 iso}, so we will be skipping some of the details.
    The image of a family of polarized K3 surfaces of degree $6$, whose polarization induces a birational morphism, is equal to a complete intersection of bidegree $(2,3)$ in a $\PP^4$-bundle. Moreover, this image will have at most ADE singularities.
    
    This fact determines a morphism $f:\Ucal_6\to\FF(2,3,4)\rat$, which we claim to be an isomorphism. For this, we construct the inverse map explicitly: this is given by taking an object $(X\subset P\to S)$ of $\FF(2,3,4)\rat(S)$ and, up to passing to an \'etale cover of $S$ so to have a line bundle $\Ocal(1)$ on $P$, sending this object to $(X\to S, \Ocal(1)|_X)$, which belongs to $\Ucal_6(S)$.

    
    
    
\end{proof}

\mypar \label{par:Pic(U6)} The complement of the substack $\FF(2,3,4)\rat$ inside $\FF(2,3,4)$ has codimension $>1$: indeed, let $U\rat$ be the preimage of $\FF(2,3,4)\rat$ along the $\PGL_{5}$-torsor $(\PP(V_{2,3})\setminus Z)\to \FF(2,3,4)$. By definition, its complement is strictly contained in the closed subscheme of singular complete intersections, which is an irreducible divisor, as observed at the end of the proof of \ref{prop:FF is a quotient}.

We deduce that $\Pic(\FF(2,3,4)\rat)\simeq\Pic(\FF(2,3,4))$. \propref{prop:Pic FF} implies that the latter is isomorphic to the sublattice of $\Pic(\PP(V_{2,3}))\simeq \ZZ^{\oplus 2}$ below:
\[ \Pic(\FF(2,3,4))\simeq \ZZ\cdot[\pi^*\Ocal_{\PP(W_2)}(5)]\oplus \ZZ\cdot [\Ocal_{\PP(V_{2,3})}(1)\otimes\pi^*\Ocal_{\PP(W_2)}(1)] \]
\corref{cor:Pic FF 2} allows us to give an explicit description of the generators of $\Pic(\Ucal_6)$: the first one is equal to the line bundle $\Phi_{0,5}$ (we are using the notation introduced in \parref{par:hodge for FF}) and the second one is just the Hodge line bundle restricted to $\Ucal_6$.

We are ready to state and prove the main result of this subsection.

\begin{thm}\label{thm:Pic(K6)}
We have:
\begin{enumerate}
    \item $\Pic(\Ucal_6)=\ZZ\cdot [\Dcal_{3,1}|_{\Ucal_6}]\oplus \ZZ\cdot \lambda_1|_{\Ucal_6}$.
    \item $\Pic(\Mcal_6)\simeq \ZZ\cdot [\Dcal_{1,1}] \oplus \ZZ\cdot [\Dcal_{2,1}] \oplus \ZZ\cdot [\Dcal_{3,1}] \oplus \ZZ\cdot \lambda_1 $.
\end{enumerate}
\end{thm}
\begin{proof}
    \propref{prop:proj model} tells us that $\Ucal_6=\Mcal_6\setminus (\Dcal_{1,1}\cup\Dcal_{2,1})$.
    
    By \propref{prop:U6 iso}, we know that $\Pic(\Ucal_6)\simeq \Pic^{\PGL_5}(\PP(V_{2,3})$. The latter had been computed in \parref{par:Pic(U6)}
    
    Observe that the divisor $\Delta$ of singular quadrics in $\PP(W_2)$ has degree $5$: this follows from the formula for the degree of the resultant. Its preimage in $\PP(V_{2,3})$ is then a generator of the $\PGL_5$-equivariant Picard group of $\PP(V_{2,3})$.
    
    The isomorphism of \propref{prop:U6 iso} induces an isomorphism between $\Dcal_{3,1}|_{\Ucal_4}$ and complete intersections of singular quadrics and cubics.
    
    This can be seen as follows: let $\overline{X}\subset \PP^4$ be a complete intersection of a quadric and a cubic containing a curve $C$ of genus $1$ and degree $3$. A simple computation with Riemann-Roch formula shows that this curve must be planar, i.e. is contained in some $\PP^2\subset\PP^4$.
    
    The intersection $\PP^2\cap \overline{X}$ is a conic, unless the complete intersection contains a plane, which should then be the case.
    
    Therefore, we can assume that the degree $2$ homogeneous form defining the quadric is contained in the homogeneous ideal generated by $X_3$ and $X_4$. It is pretty straightforward to check that the determinant of any such quadratic form is zero.
    
    We deduce the following exact sequence:
    \[ \oplus_{i=1}^{3} \ZZ\cdot [\Dcal_{i,1}] \to \Pic(\Mcal_6)\to \ZZ\cdot\lambda_1 \cdot 0 \]
    The arrow on the left must be injective, otherwise the rank of $\Pic(\Mcal_6)$ would be $<4$, and we know from \cite{Bru}*{(6)} that the sublattice of Noether-Lefschetz divisors has rank $\geq 4$.
    
    This implies the theorem.
\end{proof}


\subsection{Degree eight case}
 
\mypar Let $(X\to S,\sigma)$ be family of polarized K3 surfaces with at most rational double points of degree $8$. Let $\psi_{\sigma}:X\to P$ be the morphism induced by the polarization $\sigma$: the fibres of $P$ over $S$ are hence isomorphic to $\PP^5$. 

Let us define $\Ucal_8$ as the open substack inside $\Mcal_8$ of families such that the morphism $\psi_{\sigma}:X\to P$ is birational: \propref{prop:proj model} implies that $\Ucal_8=\Mcal_8\setminus \left(\Dcal_{1,1}\cup \Dcal_{2,1}\right)$.

\begin{prop}\label{prop:U8 iso}
    The stack $\Ucal_8\setminus \Dcal_{3,1}$ is isomorphic to $\GG(2,3,5)\rat$.
\end{prop}
\begin{proof}
    The proof is similar to the one of \propref{prop:U4 iso}: if $(\pi:X\to S,\sigma)$ is a family of polarized K3 surfaces with at most rational double points such that $\psi_{\sigma}:X\to P$ is birational, the image $\psi_{\sigma}(X)$ will be a family of complete intersection of three quadrics with at most rational double points, unless there exists a curve $C\subset X$ of genus $1$ and degree $3$, where the degree is computed with respect to $\sigma$. For a proof of this, see \cite{SD}*{Thm. 7.2}.
    
    Henceforth, we obtain a morphism $f:\Ucal_8\setminus\Dcal_{3,1}\to \GG(2,3,5)\rat$. On the other hand, given an object $(X\subset P\to S)$, up to passing to an \'etale cover of $S$, we have that $P\simeq \PP(E)$ for a vector bundle $E\to S$, hence we can take the pair $(X\to S,\Ocal(1)|_X)$, which is by construction an object of $\Ucal_8\setminus \Dcal_{3,1}$. The morphism of stacks $\GG(2,3,5)\rat \to \Ucal_8\setminus\Dcal_{3,1}$ defined in this way gives an inverse to $f$.
    
        
\end{proof}

\begin{cor}
    The stack $\GG(2,3,5)\rat$ is smooth, open in $\GG(2,3,5)$ and its complement has codimension $>1$.
\end{cor}

\begin{proof}
    By \propref{prop:U8 iso} we know that $\GG(2,3,5)\rat$ is smooth and has the dimension of $\GG(2,3,5)$, which is irreducible: we deduce that $\GG(2,3,5)\rat$ is open in $\GG(2,3,5)$.
    
    Its complement is strictly contained in the irreducible divisor (see the proof of \propref{prop:GG is a quotient}) of singular complete intersection, hence must have codimension $>1$.
\end{proof}

\mypar \label{par:Phi8} Let $E$ be the standard $\GL_6$-representation, and define $W_2:=\Sym^2 E^{\vee}$. We know from \propref{prop:Pic GG} that $\Pic(\GG(2,3,5)$ is isomorphic to the $\PGL_6$-equivariant Picard group of $\Gr_3(W_2)$.

This group is generated by the line bundle $\det(\Tcal)$ with the $\PGL_6$-linearization given by the formula:
\[ A\cdot q_1\wedge q_2\wedge q_3:=\det(A) q_1(A^{-1}\underline{x})\wedge q_2(A^{-1}\underline{x}) \wedge q_3(A^{-1}\underline{x}) \]
The quotient of $\det(\Tcal)$ by $\PGL_6$ gives a generator for $\Pic(\GG(2,3,5)$.

\propref{prop:U8 iso} implies that $\Pic(\Ucal_8\setminus\Dcal_{3,1})\simeq \Pic(\GG(2,3,5)$, which we computed in \propref{prop:Pic GG}.

By \corref{cor:Pic GG 2}, we see that a generator for $\Pic(\Ucal_8\setminus \Dcal_{3,1})$ is given by restriction of the Hodge line bundle.

We are ready to prove the main result of the subsection.

\begin{thm}\label{thm:Pic K8}
    We have:
    \begin{enumerate}
        \item $\Pic(\Ucal_8\setminus \Dcal_{3,1})\simeq \ZZ\cdot\lambda_1|_{\Ucal_8\setminus\Dcal_{3,1}}$ .
        \item $\Pic(\Mcal_8)\simeq \ZZ\cdot [\Dcal_{1,1}]\oplus\ZZ\cdot [\Dcal_{2,1}]\oplus \ZZ\cdot [\Dcal_{3,1}]\oplus \ZZ\cdot \lambda_1$.
    \end{enumerate}
\end{thm}
\begin{proof}
    We proved point (1) in \parref{par:Phi8}. The localization exact sequence applied to $(\cup_{i=1}^3 D_{i,1})\subset \Mcal_8$ gives:
    \[ \oplus_{i=1}^3 \ZZ\cdot [\Dcal_{i,1}]\to \Pic(\Mcal_8)\to\ZZ\cdot\lambda_1|_{\Ucal_8\setminus\Dcal_{3,1}}\to 0\]
    The arrow on the left must be injective, otherwise we would have ${\rm rk}(\Pic (\Mcal_8))<4$, and the rank of the sublattice of $\Pic(\Mcal_8)$ spanned by Noether-Lefschetz divisors is known to be equal to $4$ by \cite{Bru}*{(6)}.
    
    From this (2) easily follows.
\end{proof}


\section{Some related computations}\label{sec:computations}

In this last section we compute the cycle classes of some divisors in the moduli stack of complete intersections. These results are then interpreted in terms of Noether-Lefschetz divisors in the stack of polarized K3 surfaces with at most rational double points.

More precisely, we compute in \propref{prop:Pic(GGsm)} the class of $[\GG(d,m,n)_{\rm sing}]$, the divisor of singular complete intersections, in terms of the generator $\Phi$ (see \parref{par:hodge for GG}) of $\Pic(\GG(d,m,n))$.

We deduce in \propref{prop:D004} an expression of $[\Dcal_{0,0}|_{\Ucal_4}]$ in terms of $\lambda_1|_{\Ucal_4}$, where $\Ucal_4$ is the stack of polarized K3 surfaces of degree $4$ whose polarization induces a birational morphism.

We also deduce in \propref{prop:D008} an expression for $[\Dcal_{0,0}|_{(\Ucal_8\setminus\Dcal_{3,1})}]$.

A similar computation is also performed for the stack $\FF(a,b,n)$, and the results obtained therein are used to write the class of $[\Dcal_{0,0}|_{\Ucal_6}]$ in terms of $[\Dcal_{3,1}|_{\Ucal_6}]$ and $\lambda_1|_{\Ucal_6}$ (see \propref{prop:D006}).

Finally, in \propref{prop:D31 bis} we compute the cycle class in the Picard group of $\GG(4,1,3)$ of the divisor of quartic surfaces containing a line. From this we deduce in \propref{prop:D31} an expression of $[\Dcal_{3,1}|_{\Ucal_4}]$ in terms of the generator $\lambda_1|_{\Ucal_4}$ of $\Pic(\Ucal_4)$.

\subsection{The divisor of singular complete intersections in $\GG(d,m,n)$}

\mypar Recall from \parref{par:GG} that $\GG(d,m,n)$ is the stack of complete intersections of $m$ hypersurfaces of degree $d$ in a Brauer-Severi variety having $n$-dimensional fibres.

\propref{prop:GG is a quotient} tells us that
\[ \GG(d,m,n)\simeq [(\Gras\setminus Z)\setminus \PGL_{n+1}] \]
where $W_d=\Sym^d E^{\vee}$, the $d$-th symmetric power of the dual of the standard $\GL_{n+1}$-representation, and $Z$ is a $\PGL_{n+1}$-invariant, closed subscheme of codimension $>1$.

Let $\GG(d,m,n)_{\rm sing}$ be the closed substack of singular complete intersections. The isomorphism of \propref{prop:GG is a quotient} sends $\GG(d,m,n)$ to the $\PGL_{n+1}$-quotient of the divisor $\Gras_{\rm sing}$ of singular complete intersections embedded in $\PP^n$.

The expression of $[\Gras_{\rm sing}]$ in terms of $\det(\Tcal)^{\vee}$ is given by the following Proposition.

\begin{prop}\label{prop:formula Gras^sing}
    We have:
    \[ [\Gras_{\rm sing}]=\left(\sum_{i=0}^{n-m+1}  (-1)^i\binom{n+1}{i}\binom{n+1-i}{m}d^{n-i}\right)\cdot [\det(\Tcal)^{\vee}]\]
\end{prop}
\begin{proof}
    The proof is basically the same as the proof of \cite{Dil18}*{Cor. 2.2.9}, with the only difference that one has to put $c_1=0$ in the formula contained therein.
\end{proof}

Let $k:=mcm(n+1,md)/md$. Then by \corref{cor:Pic GG 2} we know that the Picard group of $\GG(d,m,n)$ is freely generated by the line bundle $\Phi$, which is equal to the $\PGL_{n+1}$-quotient of the line bundle $[\det(\Tcal)^{\otimes (-k)}]$ defined over $\Gras$.

Let $\GG(d,m,n)\sm$ be the substack of smooth complete intersections.

\begin{prop} \label{prop:Pic(GGsm)}
Let $k:=mcm(n+1,md)/md$. Then $\Pic(\GG(d,m,n)\sm)$ is generated by $\Phi$ and it is cyclic of order
    \[\frac{1}{k}\left(\sum_{i=0}^{n-m+1} (-1)^i\binom{n+1}{i}\binom{n+1-i}{m}d^{n-i}\right)\]
\end{prop}
\begin{proof}
    It follows from the short exact sequence
    \[ 0\to\ZZ\cdot[\GG(d,m,n)_{\rm sing}]\to\Pic(\GG(d,m,n))\to\Pic(\GG(d,m,n)\sm)\to 0 \]
    that $\Pic(\GG(d,m,n)\sm)$ is generated by $\Phi$ and it is cyclic of order equal to the number $p$ such that $[\GG(d,m,n)_{\rm sing}]=p\cdot\Phi$.
    
    Observe that $[\Gras_{\rm sing}]_{\PGL_{n+1}}=p\cdot [\det(\Tcal)^{\otimes (-k)}]_{\PGL_{n+1}}$.
    
    The formula of \propref{prop:formula Gras^sing} allows us to compute $pk$, hence $p$.
\end{proof}

\mypar Let $\Ucal_4$ be the open substack of $\Mcal_4$ having as objects the families of polarized K3 surfaces with at most rational double points $(\pi:X\to S,\sigma)$ whose polarization induces a birational morphism.

By \propref{prop:U4 iso} we have:
\[ \Ucal_4\simeq\GG(4,1,3)\rat \]
where the stack on the right is the stack of complete intersections with at most rational double points.

Consider the Noether-Lefschetz divisor $\Dcal_{0,0}|_{\Ucal_4}$: the objects of this substack correspond to family of K3 surfaces with at least one singular point. We deduce that $\Dcal_{0,0}|_{\Ucal_4}$ is isomorphic to the restriction of the divisor $[\GG(4,1,3)_{\rm sing}]$.

\begin{prop}\label{prop:D004}
    $[\Dcal_{0,0}|_{\Ucal_4}]=108\cdot\lambda_1|_{\Ucal_4}$
\end{prop}
\begin{proof}
    A straightforward application of \propref{prop:Pic(GGsm)}.
\end{proof}

\mypar A similar argument can be applied in the case of polarized K3 surfaces of degree $8$. More precisely, let $\Ucal_8^o=\Ucal_8\setminus\Dcal_{3,1}$ be the substack of families whose polarization induces a birational morphism with a complete intersection.

The isomorphism of \propref{prop:U8 iso} sends the restricted Noether-Lefschetz divisor $\Dcal_{0,0}|_{\Ucal_8^o}$ to an open substack of ${\rm Gr}_3(W_2)$, whose class we know from the formula contained in \propref{prop:formula Gras^sing}.

\begin{prop}\label{prop:D008}
    $[\Dcal_{0,0}|_{\Ucal_8^o}]=80\cdot\lambda_1|_{\Ucal_8^o}$
\end{prop}

\subsection{The divisor of singular complete intersections in $\FF(a,b,n)$}

\mypar Let $\PP(V_{a,b})$ be the projective bundle over $\PP(W_a)$, the projective space of degree $a$ hypersurfaces in $\PP^n$, that we introduced in \parref{par:P(Vab)}. We can interpret this scheme as a paramenter space for complete intersections in $\PP^n$ of codimension $2$ and bidegree $(a,b)$.

Let $\PP(V_{a,b})_{\rm sing}$ be the closed subscheme of singular complete intersections in $\PP^n$ of bidegree $(a,b)$. We briefly sketched in the proof of \propref{prop:FF is a quotient} that $\PP(V_{a,b})_{\rm sing}$ is an irreducible divisor.

Recall that $\Pic(\PP(V_{a,b})$ is freely generated by $\pi^*\Ocal_{\PP(W_a)}(1)$ and $\Ocal_{\PP(V_{a,b})}(1)$. The class of $\PP(V_{a,b})_{\rm sing}$ is given by the following formula.

\begin{prop}\label{prop:P(Vab)sing}
Let 
\begin{align}
    &A=\sum_{i=0}^{n-1}\sum_{k=0}^{n-1-i} (-1)^i\binom{n+1}{i}a^{n-1-i-k}b^{k} \\
    &B=\sum_{i=0}^{n-2}\sum_{k=0}^{n-2-i} (-1)^i\binom{n+1}{i}(n-1-i-k)a^{n-2-i-k}b^k \\
    &C=\sum_{i=0}^{n-2}\sum_{k=1}^{n-1-i} (-1)^i\binom{n+1}{i}ka^{n-1-i-k}b^{k-1}
\end{align}
Then:
\[ [\PP(V_{a,b})_{\rm sing}]=(ab\cdot B + b\cdot A)[\pi^*\Ocal_{\PP(W_a)}(1)]+(ab\cdot C+ a\cdot A)[\Ocal_{\PP(V_{a,b})}(1)] \]
\end{prop}
\begin{proof}
    The proof is the same as the one of \cite{Dil18}*{Prop. 1.2.6}, with the only difference that in our context $c_1=0$.
\end{proof}

\mypar Recall that $\FF(a,b,n)$ is the stack of complete intersections of two hypersurfaces of degree $a$ and $b$ inside a Brauer-Severi variety having $n$-dimensional fibres.

In \propref{prop:FF is a quotient} we established an isomorphism of this stack with the quotient stack $[(\PP(V_{a,b})\setminus Z)/\PGL_{n+1}]$, where $Z$ is a $\PGL_{n+1}$ invariant closed subscheme of $\PP(V_{a,b})$.

Let $\FF(a,b,n)_{\rm sing}$ be the closed substack of singular complete intersections: this is sent by the isomorphism of \propref{prop:FF is a quotient} to a restriction of the substack $[\PP(V_{a,b})_{\rm sing}/\PGL_{n+1}]$.

Therefore, the isomorphism of Picard groups
\[ \Pic(\FF(a,b,n)\simeq\Pic^{\PGL_{n+1}}(\PP(V_{a,b}) \]
sends $[\FF(a,b,n)_{\rm sing}]$ to the $\PGL_{n+1}$-equivariant class of $\PP(V_{a,b})_{\rm sing}$. From this we deduce the following result.

\begin{prop} \label{prop:Pic(FFsm)}
    Let $\FF(a,b,n)\sm$ be the stack of smooth complete intersections of bidegree $(a,b)$ in a Brauer-Severi variety having $n$-dimensional fibres.
    
    The $\Pic(\FF(a,b,n)\sm)$ is a quotient lattice of the free rank $2$ lattice $\Pic(\FF(a,b,n))$ (see \propref{prop:Pic FF}) by the rank $1$ lattice spanned by $[\PP(V_{a,b})_{\rm sing}]$.
\end{prop}

\mypar Let $\Ucal_6$ be the stack of polarized K3 surfaces with at most rational double points of degree $6$ whose polarization induces a birational morphism.

As in the previous cases, the class of the Noether-Lefschetz divisor $\Dcal_{0,0}|_{\Ucal_6}$ is sent by the isomorphism of Picard groups
\[ \Pic(\Ucal_6)\simeq\Pic(\FF(2,3,4)) \]
induced by \propref{prop:U6 iso} to the cycle class of $\FF(2,3,4)_{\rm sing}$. Applying \propref{prop:Pic(FFsm)} we obtain the following result.

\begin{prop}\label{prop:D006}
    $[\Dcal_{0,0}|_{\Ucal_6}]=78\cdot [\Dcal_{3,1}|_{\Ucal_6}] + 98\cdot \lambda_1|_{\Ucal_{6}}$.
\end{prop}

\subsection{Computation of $[(\Dcal_{3,1})|_{\Ucal_4}]$}

\mypar Let $\Ucal_4$ be the stack of polarized K3 surfaces with at most rational double points of degree $4$ whose polarization induces a birational morphism.

Consider the Noether-Lefschetz divisor $\Dcal_{3,1}$, which by definition consists of those polarized surfaces that contain a genus $1$ curve $C$ whose degree, with respect to the polarization, is equal to $3$.

\propref{prop:U4 iso} together with \propref{prop:Pic GG} gives an isomorphism:
\[ \Pic(\Ucal_4)\simeq\Pic^{\PGL_4}(\PP(W_4)) \]
This isomorphism sends $[\Dcal_{3,1}|_{\Ucal_4}]$ to the $\PGL_4$-equivariant cycle class of $D_{3,1}$, the divisor of degree $4$ hypersurfaces in $\PP^3$ containing a genus $1$ curve of degree $3$.

\begin{prop}\label{prop:D31 bis}
    $[D_{3,1}]=[\Ocal_{\PP(W_4)}(320)]$ in $\Pic(\PP(W_4))$.
\end{prop}

From this we deduce the following:

\begin{prop}\label{prop:D31}
    $[\Dcal_{3,1}]=320\cdot\lambda_1|_{\Ucal_4}$ in $\Pic(\Ucal_4)$.
\end{prop}
\begin{proof}
Recall from \propref{prop:Pic GG} that $\Pic^{\PGL_4}(\PP(W_4))$ is generated by the line bundle $\Ocal_{\PP(W_4)}(1)$, which descends to the class of the Hodge line bundle $\lambda_1|_{\Ucal_4}$.

The class of $[(\Dcal_{3,1})|_{\Ucal_4}]$ is sent by the isomorphism of Picard groups 
\[ \Pic(\Ucal_4)\simeq \Pic^{\PGL_4}(\PP(W_4)) \]
to the equivariant cycle class of $D_{3,1}$, which we computed in \propref{prop:D31 bis}.
\end{proof}

The remainder of this subsection is devoted to prove \propref{prop:D31 bis}. We start with a technical lemma. In what follows, we will denote $CH(X)$ the Chow group/ring of a variety $X$.

\begin{lm}\label{lm:class proj sub}
     Let $0\to F\to E\to Q\to 0$ be a short exact sequence of vector bundles or rank $f$, $e$ and $q$ over a smooth variety $X$. Then we have:
     \[ [\PP(F)]=\sum_{i=0}^{q} c_i(Q)\cdot t^{q-i} \]
     inside $CH(\PP(E))$, where $t$ denotes the hyperplane class of $\PP(E)$.
\end{lm}
\begin{proof}
    Let $p:\PP(E)\to B$ be the projection morphism. Then the fundamental class of $\PP(F)$ is given by the top Chern class of the normal bundle, which is equal to $p^*Q\otimes\Ocal_{\PP(E)}(1)$.
    
    A simple computation with the Chern roots of the normal bundle prove the formula above.
\end{proof}
The first step towards a proof of \propref{prop:D31 bis} is the following.

\begin{lm}
     The divisor $D_{3,1}$ in $\PP(W_4)$ is equal to the divisor of quartic surfaces containing a line.
\end{lm}
\begin{proof}
    Let $X$ be hypersurface in $D_{3,1}$. Call $C$ the genus $1$ curve in $X$: then $C$ must be contained in a hyperplane $H\subset\PP^2$. We can assume that the intersection of $H$ with $X$ is a curve of degree $4$ that must also contain $C$: from this it follows that there is a line $L$ inside $X$.
    
    On the other hand, if $X$ contains a line $L$, we can always find a hyperplane $H$ containing $L$ and not contained in $X$: the intersection $H\cap X$ must then be equal to the union of $L$ with another curve $C$ of genus $1$ and degree $3$.
\end{proof}
Before proceeding further, we need some preliminary results on the geometry of quartic hypersurfaces containing a line.
 
\begin{lm}\label{lm:unique line}
     The generic degree $4$ hypersurface in $\PP^3$ that contains a line actually contains only one line.
\end{lm}
\begin{proof}
    The generic degree $4$ surface $X$ in $\PP^3$ containing a line $F$ has $\Pic(X)\simeq \ZZ\cdot[F]\oplus\ZZ\cdot [L]$, where $L$ is the polarization (we can assume it to be very ample). Recall that $(L^2)=4$, $(F\cdot L)=1$ and $(F^2)=-2$.
    
    Therefore, any other line $F'$ contained in $X$ must be linearly equivalent to a divisor of the form $a\cdot L+b\cdot F$ for some integers $a$, $b$. A straightforward computation shows that we must necessarily have $a=0$, $b=1$. In other terms $F'$ must be linearly equivalent to $F$. 
    
    We can apply Riemann-Roch to compute $h^0(\Ocal_X(F))$: this number turns out to be $1$, hence $F'$ is actually equal to $F$.
\end{proof}
    
\mypar Let $\Gr_1(\PP^3)$ be the grassmannian of projective lines in $\PP^3$ and call $\Tcal$ the associated rank $2$ tautological bundle. Let $\Jcal$ be the ideal sheaf of the closed embedding $\PP(\Tcal)\hookrightarrow \PP^3\times\Gr_1(\PP^3)$, so that we have the following exact sequence on $\PP^3\times\Gr_1(\PP^3)$:
    \[ 0\to\Jcal\to\Ocal_{\PP^3\times\Gr_1(\PP^3)}\to \Ocal_{\PP(\Tcal)}\to 0\]
    We twist the sequence above by $\pr_1^*\Ocal_{\PP^3}(4)$ and we take the pushforward along $\pr_2$, obtaining in this way the following exact sequence on $\Gr_1(\PP^3)$:
    \[ 0\to \pr_{2*}(\Jcal\otimes\pr_1^*\Ocal_{\PP^3}(4))\to W_4\times\Ocal_{\Gr_1(\PP^3)}\to \pr_{2*}(\Ocal_{\PP(\Tcal)}\otimes\pr_1^*\Ocal_{\PP^3}(4))\]
    We claim that the morphism on the right is surjective, i.e. $R^1\pr_{2*}(\Jcal\otimes\pr_1^*\Ocal_{\PP^3}(4))=0$. 
    
    This vanishing can be deduced applying cohomology and base change theorem (\cite{Hart}*{3.12.11}), once we prove that $H^1(\PP^3,J(4))=0$, where $J$ is the ideal of a line in $\PP^3$.
    
    Consider the Koszul resolution of $J$:
    \[0\to\Ocal_{\PP^3}(-2)\to\Ocal_{\PP^3}(-1)^{\oplus 2}\to J\to 0\]
    By twisting with $\Ocal_{\PP^3}(4)$ and taking the associated long exact sequence in cohomology, one easily verifies that $H^1(\PP^3,J(4))=0$.
    
    Let $\Fcal:=\pr_{2*}(\Jcal\otimes\pr_1^*\Ocal_{\PP^3}(4))$: then the short exact sequence above defines a closed embedding of $\PP(\Fcal)$ inside $\PP(W_4)\times\Gr_{1}(\PP^3)$. Moreover, the projection $\PP(\Fcal)\to\PP(W_4)$ is generically 1:1 onto $D_{3,1}$: this is because $\PP(\Fcal)$ parametrizes those pairs $(X,F)$ where $F$ is a line in $\PP^3$, $X$ is a hypersurface of degree $4$ and $F\subset X$, and by \lmref{lm:unique line} the generic hypersurface that satisfies this property contains only one line.
    
    Let $p:\PP(W_4)\times\Gr_1(\PP^3)\to\PP(W_4)$ be the projection onto the first factor: then $p_*[\PP(\Fcal)]=[D_{3,1}]$, where $[\PP(\Fcal)]$ denotes the fundamental class of $\PP(\Fcal)$ in the Chow ring $CH(\PP(W_4)\times\Gr_1(\PP^3))$.
    
\mypar The projective bundle formula tells us that \[CH(\PP(W_4)\times\Gr_1(\PP^3))\simeq CH(\Gr_1(\PP^3))[h]/(f(h))\]
    where $f$ is a relation in degree equal to the rank of $W_4$, which is $35$.
    
    Observe that the rank of $\Qcal:=\pr_{2*}(\Ocal_{\PP(\Tcal)}\otimes\pr_1^*\Ocal_{\PP^3}(4))$ is constantly equal to $5$, hence by Grauert theorem this is a locally free sheaf. Applying \lmref{lm:class proj sub} we have that 
    \[ [\PP(\Fcal)]=\sum_{i=0}^5 c_i(\Qcal)\cdot h^{5-i} \]
    hence
    \[ p_*[\PP(\Fcal)]=\sum_{i=0}^5 (p_*c_i(\Qcal))\cdot h^{5-i} \]
    The only non-zero term in the formula above is $p_*(c_4(\Qcal))\cdot h$ for fairly simple dimensional reasons, and
    \[ p_*(c_4(\Qcal))=\int_{\Gr_1(\PP^3)} c_4(\Qcal) \]
    so that we are left with computing the integral above.
    
\mypar The Chow rings of $\PP(W_4)$ and $\Gr_1(\PP^3)$ are both free as $\ZZ$-modules, and the Chow ring of $\PP(W_4)\times\Gr_1(\PP^3)$ is freely generated as $CH(\Gr_1(\PP^3))$-module. This implies that we can retrieve the Chern classes of $\Qcal$ from its Chern character $ch(\Qcal)$, using the relations:
   \[ \begin{aligned}
       c_1(\Qcal)  & = ch_1(\Qcal) \\
       c_2(\Qcal)  & = \frac{c_1(\Qcal)^2}{2}-ch_2(\Qcal) \\
       c_3(\Qcal)  & =2ch_3(\Qcal)-\frac{c_1(\Qcal)^3}{3}+c_1(\Qcal)c_2(\Qcal)\\
       c_4(\Qcal)  & =
       \frac{c_1(\Qcal)^4}{4}-c_1(\Qcal)^2c_2(\Qcal)+\frac{c_2(\Qcal)^2}{2}+c_1(\Qcal)c_3(\Qcal)-6ch_4(\Qcal)
    \end{aligned}\]
    We can use Grothendieck-Riemann-Roch formula to compute $ch(\Qcal)$.
    
    Let us recall where the locally free sheaf $\Qcal$ comes from: it is obtained as the pushforward along $\pr_2:\PP^3\times\Gr_1(\PP^3)$ of the sheaf $\Ocal_{\PP(\Tcal)}\otimes\pr_1^*\Ocal_{\PP^3}(4)$, where $\PP(\Tcal)$ is the projectivization of the tautological vector bundle over $\Gr_1(\PP^3)$.
    
    The cohomology groups $H^i(F,\Ocal_F(4))$ vanish for any line $F\subset\PP^3$ and $i>0$: using the cohomology and base change theorem we deduce that the higher direct images of $\Ocal_{\PP(\Tcal)}\otimes\pr_1^*\Ocal_{\PP^3}(4)$ along $\pr_2$ vanish as well, hence the following equality holds in $K_0(\Gr_1(\PP^3)$, the $K_0$-group of vector bundles on the grassmannian:
    \[  \pr_{2*}[\Ocal_{\PP(\Tcal)}\otimes\pr_1^*\Ocal_{\PP^3}(4)]=[\pr_{2*}(\Ocal_{\PP(\Tcal)}\otimes\pr_1^*\Ocal_{\PP^3}(4))] \]
    Therefore $ch(\Qcal)$ is equal to the Chern character of the term on the left, and by Grothendieck-Riemann-Roch we have:
    \[ ch(\pr_{2*}[\Ocal_{\PP(\Tcal)}\otimes\pr_1^*\Ocal_{\PP^3}(4)])= \pr_{2*}(ch([\Ocal_{\PP(\Tcal)}\otimes\pr_1^*\Ocal_{\PP^3}(4)])\cdot \pr_1^*Td(T_{\PP^3}))\]
    Let $t$ denote the hyperplane class of $\PP^3$ (and its pullback to $\PP^3\times\Gr_1(\PP^3)$), so that we have:
    \[ \pr_1^*Td(T_{\PP^3})=\left(\frac{t}{1-e^{-t}}\right)^4\]
    The multiplicativity of the Chern character holds:
    \[ ch([\Ocal_{\PP(\Tcal)}\otimes\pr_1^*\Ocal_{\PP^3}(4)])=ch(\Ocal_{\PP(\Tcal))}\cdot \pr_1^*ch(\Ocal_{\PP^3(4)}=ch(\Ocal_{\PP(\Tcal)})\cdot e^{4t}\]
    We apply again Grothendieck-Riemann-Roch to compute $ch(\Ocal_{\PP(\Tcal)})$, as the higher direct images of the embedding $i:\PP(\Tcal)\hookrightarrow \PP^3\times\Gr_1(\PP^3)$ vanish. We get:
    \[ ch(\Ocal_{\PP(\Tcal)})=i_*(Td([T_{\PP(\Tcal)}]-i^*[T_{\PP^3\times\Gr_1(\PP^3)}])) \]
    The short exact sequence for a regular embedding implies that the difference $[T_{\PP(\Tcal)}]-i^*[T_{\PP^3\times\Gr_1(\PP^3)}]$ is equal to $-[N_{\PP(\Tcal)}]$, where the latter is the class of the normal bundle.
    
    The normal bundle of $\PP(\Tcal)$ is equal to $i^*(\pr_2^*\Scal\otimes\pr_1^*\Ocal_{\PP^3}(1))$, where $\Scal$ is the tautological quotient bundle over $\Gr_1(\PP^2)$. We deduce:
    \[ ch(\Ocal_{\PP(\Tcal)})=Td([\pr_2^*\Scal\otimes\pr_1^*\Ocal_{\PP^3}(1)])^{-1}\cdot [\PP(\Tcal)] \]
    Let $\alpha$ and $\beta$ be the Chern roots of $\Scal$, and let $s_1=\alpha+\beta$ and $s_2=\alpha\beta$ be respectively the first and the second Chern class of $\Scal$.  The Chern roots of $\pr_2^*\Scal\otimes\pr_1^*\Ocal_{\PP^3}(1)$ are then $\alpha+t$ and $\beta+t$.
    
    \lmref{lm:class proj sub} gives us:
    \[\PP(\Tcal)=\sum_{i=0}^2 s_i\cdot t^{2-i} \]
    and by definition of Todd class, we get:
    \[ Td([\pr_2^*\Scal\otimes\pr_1^*\Ocal_{\PP^3}(1)])^{-1}=\frac{(1-e^{-\alpha-t})(1-e^{-\beta-t})}{(\alpha+t)(\beta+t)}\]
    Putting all together, we get the following formula for $ch(\pr_{2*}[\Ocal_{\PP(\Tcal)}\otimes\pr_1^*\Ocal_{\PP^3}(4)])$:
    \[ \pr_{2*}\left(e^{4t}\left(\frac{(1-e^{-\alpha-t})(1-e^{-\beta-t})}{(\alpha+t)(\beta+t)}\right)\left(\sum_{i=0}^2 s_i\cdot t^{2-i}\right)\left(\frac{t}{1-e^{-t}}\right)^4\right)
    \]
    Taking the pushforward along $\pr_{2*}$ is equivalent to taking the coefficient in front of the term $t^3$. A straightforward computation gives the following result:
    \begin{align*}
       ch(\pr_{2*}[\Ocal_{\PP(\Tcal)}\otimes\pr_1^*\Ocal_{\PP^3}(4)])_0=& 5\\
       ch(\pr_{2*}[\Ocal_{\PP(\Tcal)}\otimes\pr_1^*\Ocal_{\PP^3}(4)])_1=& 10s_1 \\
       ch(\pr_{2*}[\Ocal_{\PP(\Tcal)}\otimes\pr_1^*\Ocal_{\PP^3}(4)])_2=& -5s_1^2+20s_2 \\
       ch(\pr_{2*}[\Ocal_{\PP(\Tcal)}\otimes\pr_1^*\Ocal_{\PP^3}(4)])_3=& \frac{5}{3}s_1^3-10s_1s_2\\
       ch(\pr_{2*}[\Ocal_{\PP(\Tcal)}\otimes\pr_1^*\Ocal_{\PP^3}(4)])_4=& \frac{5}{12}s_1^4+\frac{10}{3}s_1^2s_2-\frac{5}{3}s_2^2 
    \end{align*}
    We can then retrieve the Chern classes of $\Qcal=\pr_{2*}(\Ocal_{\PP(\Tcal)}\otimes\pr_1^*\Ocal_{\PP^3}(4))$ using the relations we mentioned before. We get:
    \begin{align*}
        c_1(\Qcal) &= 10s_1 \\
        c_2(\Qcal) &= 55s_1^2-20s_2 \\
        c_3(\Qcal) &= 220(s_1^3-s_1s_2) \\
        c_4(\Qcal) &= 5(143s_1^4-264s_1^2s_2+42s_2^2
    \end{align*}
    The class of $[D_{3,1}]$ is then equal to $[\Ocal_{\PP(W_4)}(l)]$, where $l$ is value given by the integral
    \[ \int_{\Gr_{1}(\PP^3)}5(143s_1^4-264s_1^2s_2+42s_2^2) \]
    The intersection theory of the $4$-dimensional grassmannian $\Gr_1(\PP^3)$ is quite well known since Schubert himself. In particular, if $[p]$ is the fundamental class of a point in $\Gr_1(\PP^3)$, we have:
    \begin{align*}
        s_1^4= &2[p]\\
        s_1^2s_2=s_2^2= &[p]
    \end{align*}
    From this it follows that the integral above is equal to $320$, hence $[D_{3,1}]$ is isomorphic to $\Ocal_{\PP(W_4)}(320)$. This proves \propref{prop:D31 bis}.

\end{document}